\documentclass[]{amsart}

\usepackage{graphicx,color,hyperref}
\numberwithin{equation}{section}

\usepackage[initials,nobysame]{amsrefs}
\usepackage{amssymb}
\usepackage{subcaption}

\usepackage{mathtools}
\mathtoolsset{showonlyrefs}

\newtheorem{theorem}{Theorem}[section]
\newtheorem{proposition}[theorem]{Proposition}
\newtheorem{corollary}[theorem]{Corollary}
\newtheorem{lemma}[theorem]{Lemma}

\theoremstyle{definition}

\newtheorem*{acknowledgements}{Acknowledgements}

\theoremstyle{remark}
\newtheorem{remark}[theorem]{Remark}
\newtheorem{example}[theorem]{Example}

\DeclareMathOperator{\sech}{sech}
\DeclareMathOperator{\cn}{cn}

\DeclareMathOperator{\sn}{sn}

\newcommand{\R}{\mathbf{R}}

\renewcommand{\S}{\mathbf{S}}
\newcommand{\X}{\mathbf{X}}
\newcommand{\A}{\mathcal{A}}
\newcommand{\bA}{\mathbf{A}}

\newcommand{\B}{\mathcal{B}}
\renewcommand{\L}{\mathcal{L}}
\newcommand{\E}{\mathcal{E}}

\begin{document}

\title{Smooth compactness of elasticae}
\author[T.~Miura]{Tatsuya Miura}
\address[T.~Miura]{Department of Mathematics, Graduate School of Science, Kyoto University, Kitashirakawa Oiwake-cho, Sakyo-ku, Kyoto 606-8502, Japan}
\email{tatsuya.miura@math.kyoto-u.ac.jp}
\date{\today}

\begin{abstract}
We prove a smooth compactness theorem for the space of elasticae, unless the limit curve is a straight segment. As an application, we obtain smooth stability results for minimizers with respect to clamped boundary data.
\end{abstract}

\maketitle


\section{Introduction}

Elastica theory is fairly classical, dating back to early modern times \cite{Lev,Tru83}, but some of its fundamental properties are still missing in the literature.

In this paper we address a compactness problem for the space of elasticae under natural boundedness assumptions.
This problem, despite its fundamental nature, proves to be somewhat subtle.
Our main result shows that smooth compactness holds true generically, but may fail in the exceptional case that the limit curve is a straight segment.
As an application, we clarify necessary assumptions so that smooth stability results hold for minimizers of clamped boundary value problems.
This result is directly pertinent to describing physical stability of elastic rods or surfaces with respect to boundary data.

\subsection{Smooth compactness}

Let $n\geq2$ and $I=(0,1)$.
An immersed curve $\gamma\in W^{2,2}(I;\R^n)$ is called an \emph{elastica} if it is a critical point of the \emph{bending energy}
\[
\B[\gamma]:=\int_I|\kappa|^2ds
\]
among curves with fixed length $\L[\gamma]:=\int_I ds$.
Here $ds$ denotes the arclength measure $ds:=|\partial_x\gamma|dx$, and $\kappa:=\gamma_{ss}$ denotes the curvature vector, where the arclength derivative is defined by $\partial_s:=|\partial_x\gamma|^{-1}\partial_x$.
Recall that the curve $\gamma$ is an elastica if and only if there is a multiplier $\lambda\in\R$ such that $\nabla\B[\gamma]+\lambda\nabla\L[\gamma]=0$.
By standard regularity arguments (see e.g.\ \cite{Miura2024survey}) any elastica is analytic (up to constant-speed reparametrization) and solves the Euler--Lagrange equation in the classical sense:
\begin{equation}\label{eq:EL_elastica}
    2\nabla_s^2\kappa + |\kappa|^2\kappa - \lambda \kappa=0,
\end{equation}
where $\nabla_s$ denotes the normal derivative $\nabla_sX:=X_s^\perp=X_s-\langle X_s,\gamma_s\rangle\gamma_s$.

Here is our main result on compactness of elasticae.

\begin{theorem}\label{thm:main_compactness}
    Let $\{\gamma_j\}_{j=1}^\infty\subset W^{2,2}(I;\R^n)$ be a sequence of elasticae such that
    \begin{equation}\label{eq:main_assumption}\tag{A}
        \text{there exists $C>0$ such that $\B[\gamma_j]\leq C$ and $\tfrac{1}{C}\leq \L[\gamma_j] \leq C$ for all $j$.}
    \end{equation}
    Let $\bar{\gamma}_j\in C^\omega(\bar{I};\R^n)$ denote the constant-speed reparametrization of $\gamma_j$.
    Then there are translation vectors $b_j\in\R^n$ such that the sequence $\{\bar{\gamma}_j+b_j\}$ contains a subsequence $\{\bar{\gamma}_{j'}+b_{j'}\}$ converging to some constant-speed elastica $\bar{\gamma}_\infty\in C^\omega(\bar{I};\R^n)$ in the $W^{2,2}$-weak topology and the $C^1$ topology.

    In addition, if $\bar{\gamma}_\infty$ is not a straight segment, then the convergence holds in the smooth sense, that is, $\lim_{j'\to\infty}\|\bar{\gamma}_{j'}+b_{j'}-\bar{\gamma}_\infty\|_{C^m(\bar{I};\R^n)}=0$ for any $m\geq0$.
\end{theorem}

\begin{remark}
    Without the additional assumption for smooth convergence, even $W^{2,2}$-strong convergence may fail.
    We will give planar counterexamples of both curvature oscillation type in Example \ref{ex:oscillation} (see Figure \ref{fig:curvature_oscillation}) and of concentration type in Example \ref{ex:concentration} (see Figure \ref{fig:curvature_concentration}).
    In particular, those examples imply that adding an assumption of the form $\inf_j\B[\gamma_j]\geq\frac{1}{C}$ cannot prevent nonsmooth convergence.
\end{remark}

\begin{figure}[htbp]
 \begin{center}
     \includegraphics[scale=0.4]{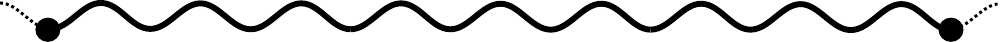}
     \caption{Counterexample of curvature oscillation type.}
     \label{fig:curvature_oscillation}
 \end{center}
\end{figure}

\begin{figure}[htbp]
 \begin{center}
     \includegraphics[scale=0.4]{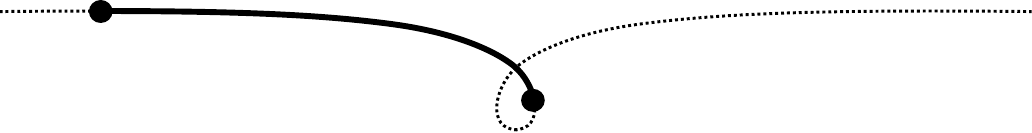}
     \caption{Counterexample of curvature concentration type.}
     \label{fig:curvature_concentration}
 \end{center}
\end{figure}

\begin{remark}\label{rem:invariance}
    If the original sequence $\{\gamma_j\}$ is also bounded in $L^\infty$, or more generally $\sup_j\min_{x\in\bar{I}}|\gamma_j(x)|<\infty$, then there is no need to take translation vectors $\{b_j\}$, since in this case $\{b_j\}$ must be bounded and thus have a convergent subsequence.
\end{remark}

We briefly discuss the idea of the proof of Theorem \ref{thm:main_compactness}.
The first assertion is a simple consequence of standard weak compactness.
The main point is when this weak convergence improves smooth convergence.
Here we focus on the behavior of the corresponding sequence of multipliers $\{\lambda_j\}$.
In fact, we will prove the following dichotomy:
\begin{itemize}
    \item If $\sup_j|\lambda_j|<\infty$, then the convergence is smooth (Proposition \ref{prop:bounded_lambda}).
    \item If $\sup_j|\lambda_j|=\infty$, then the limit is a segment (Proposition \ref{prop:unbounded_lambda}).
\end{itemize}
The first part is now reduced to standard interpolation arguments, while the latter involves a singular limit and requires more delicate arguments concerning the moduli space of elasticae.

A major benefit of Theorem \ref{thm:main_compactness} is that, under \eqref{eq:main_assumption}, smooth convergence follows by only checking a purely geometric property, without any information on the multipliers.
This point is particularly important when studying fixed-length problems, in which multipliers are not a priori controlled.

\subsection{Smooth stability for boundary value problems}

Theorem \ref{thm:main_compactness} can be used to show fundamental stability results for elasticae with respect to perturbations of parameters in various boundary value problems.
Here we focus on energy-minimal elasticae subject to the clamped boundary condition.

Let $\X:=\R^n\times\R^n\times\S^{n-1}\times\S^{n-1}$, where $\S^{n-1}\subset\R^n$ denotes the unit sphere.
Given parameters $\Gamma=(P_0,P_1,V_0,V_1)\in\X$ and $L>0$, we define the admissible set (of constant-speed curves) by
\begin{equation*}
    \mathcal{A}_{\Gamma,L} := \left\{ \gamma\in W^{2,2}(I;\R^n) \,\left|\, 
    \begin{aligned}
        &\gamma(0)=P_0,\ \gamma(1)=P_1,\\
        &\gamma_s(0)=V_0,\ \gamma_s(1)=V_1,\ |\partial_x\gamma|\equiv L\,
    \end{aligned}
    \right\}\right..
\end{equation*}
Let
\begin{equation*}
    \hat{\bA} := \{ (\Gamma,L)\in \X\times(0,\infty) \mid \A_{\Gamma,L}\neq\emptyset \}.
\end{equation*}
In view of Theorem \ref{thm:main_compactness}, we distinguish the parameters depending on whether straight segments are admissible.
It is geometrically clear that
\[
\hat{\bA}=\hat{\bA}'\cup\hat{\bA}_s,
\]
where
\begin{align*}
    \hat{\bA}' &:= \{ (\Gamma,L)\in\hat{\bA}\mid |P_1-P_0|<L\},\\
    \hat{\bA}_s &:= \{ (\Gamma,L)\in\hat{\bA}\mid |P_1-P_0|=L,\ V_0=V_1=\tfrac{P_1-P_0}{L}\}.
\end{align*}
These sets are distinguished by whether a straight segment is admissible; of course, the former set $\hat{\bA}'$ is much more important as it contains generic parameters.
Recall that for each $(\Gamma,L)\in \hat{\bA}$ there always exists a minimizer of the bending energy $\B$ in $\mathcal{A}_{\Gamma,L}$ (\cite{Miura2024survey}*{Theorem 3.1}), and such a minimizer must be a smooth elastica (\cite{Miura2024survey}*{Theorem 3.6}).
However, minimizers may not be unique for given boundary data (see \cite{Miura20}*{Figure 5}).

Here is the main smooth stability result for globally minimal elasticae.

\begin{theorem}\label{thm:stability_fixed}
    Let $(\Gamma,L)\in \hat{\bA}'$.
    Then, for any sequence $\{(\Gamma_j,L_j)\}_{j=1}^\infty\subset \hat{\bA}$ converging to $(\Gamma,L)$, and for any sequence $\{\gamma_j\}$ of minimizers $\gamma_j$ of $\B$ in $\mathcal{A}_{\Gamma_j,L_j}$, there exist a subsequence $\{\gamma_{j'}\}\subset \{\gamma_j\}$ and a minimizer $\gamma$ of $\B$ in $\mathcal{A}_{\Gamma,L}$ such that
    \[
    \lim_{{j'}\to\infty}\|\gamma_{j'}-\gamma\|_{C^m(\bar{I};\R^n)}=0 \quad \text{for all $m\geq0$.}
    \]
\end{theorem}

\begin{remark}
    Taking a subsequence is necessary due to the possible non-uniqueness of minimizers for $(\Gamma,L)$.
    Clearly, if we additionally assume that $\B$ has a unique minimizer in $\A_{\Gamma,L}$, then the smooth convergence holds without taking a subsequence.
\end{remark}

It is a major open problem in elastica theory to determine the general structure of the subset of admissible parameters $\hat{\bA}$ for which uniqueness of minimizers holds, even in the planar case $n=2$.
Some observations on uniqueness and minimality are given, e.g., in \cite{SS14,Miura20,MY_2024_Crelle,DallAcquaDeckelnick2024,miura2024uniqueness}.
In particular, in \cite{Miura20,miura2024uniqueness} it is shown that if $n=2$ and if $\Gamma=(P_0,P_1,V_0,V_1)$ satisfies the ``generic boundary angle'' assumption
\begin{equation}\label{eq:generic_angle}
    \max_{i=0,1}|\langle P_1-P_0, V_i \rangle|< |P_1-P_0|,
\end{equation}
then there is $\varepsilon_\Gamma>0$ such that for any $L>0$ with $|P_0-P_1|<L\leq |P_0-P_1|+\varepsilon_\Gamma$, the energy $\B$ admits a unique minimizer in $\A_{\Gamma,L}$.

Now we observe that our smooth stability result can even be used to improve the understanding of the structure of the ``uniqueness sets'' of boundary data.
In the author and Wheeler's recent work \cite{miura2024uniqueness}, it is shown that if a planar elastica has non-constant monotone signed curvature $k$, then it is uniquely minimal subject to the clamped boundary condition.
This monotonicity assumption may not be preserved under smooth perturbations, but the stronger property that the curvature derivative $\partial_sk$ is strictly positive or negative is always preserved (under $C^3$ perturbations).
This fact, together with Theorem \ref{thm:stability_fixed}, directly implies a nontrivial uniqueness-propagation property for certain parameters.

\begin{corollary}
    Let $n=2$ and $(\Gamma,L)\in\hat{\bA}$.
    Suppose that $\A_{\Gamma,L}$ admits an elastica whose curvature derivative has no zero.
    Then there is a neighborhood $\mathbf{U}\subset\hat{\bA}$ of $(\Gamma,L)$ such that for any $(\Gamma',L')\in\mathbf{U}$ the energy $\B$ has a unique minimizer in $\mathcal{A}_{\Gamma',L'}$.
\end{corollary}

\begin{remark}
    The same assertion also holds if we instead suppose that $\A_{\Gamma,L}$ admits a \emph{uniquely minimal} elastica whose curvature has no zero.
    This can be shown by employing the convexification argument in \cite[Section 5.3]{Miura20}.
\end{remark}

\begin{remark}
    The uniqueness-propagation property is meaningful even phenomenologically.
    Indeed, it ensures that, for example, if the minimizer for $(\Gamma,L)$ is symmetric, any perturbation of $(\Gamma,L)$ does not lead to symmetry-breaking type non-uniqueness of minimizers as in \cite[Figure 5]{Miura20}.
\end{remark}

Finally, for comparison purposes, we also obtain a similar stability result to Theorem \ref{thm:stability_fixed} for the length-penalized problem.
More precisely, given $\Gamma\in\X$ and $\lambda>0$, we minimize the \emph{modified bending energy}
\[
\E_\lambda := \B + \lambda \L
\]
among curves in the admissible set
\[
\mathcal{A}_{\Gamma}:=\bigcup_{L>0}\mathcal{A}_{\Gamma,L}.
\]
Notice that $\mathcal{A}_{\Gamma}\neq\emptyset$ for every $\Gamma\in\X$.
Here, in addition to the case that straight segments are admissible, we also independently treat the case of closed curves.
Let
\begin{align*}
    \X' := \X\setminus(\X_s\cup\X_c), \quad \X_s &:= \{ P_0\neq P_1,\ V_0=V_1=\tfrac{P_1-P_0}{|P_1-P_0|}\},\\
    \X_c &:= \{ P_0=P_1,\ V_0=V_1\}.
\end{align*}
Note also that minimizers always exist (Theorem \ref{thm:existence_penalized}) but may not be unique.
(For $\lambda\leq0$, minimizers may not even exist.)

\begin{theorem}\label{thm:stability_penalized}
    Let $(\Gamma,\lambda)\in\X'\times(0,\infty)$.
    Then, for any sequence $\{(\Gamma_j,\lambda_j)\}_{j=1}^\infty\subset \X\times(0,\infty)$ converging to $(\Gamma,\lambda)$, and for any sequence $\{\gamma_j\}$ of minimizers $\gamma_j$ of $\E_{\lambda_j}$ in $\mathcal{A}_{\Gamma_j}$, there exist a subsequence $\{\gamma_{j'}\}\subset \{\gamma_j\}$ and a minimizer $\gamma$ of $\E_\lambda$ in $\mathcal{A}_\Gamma$ such that
    \[
    \lim_{j'\to\infty}\|\gamma_{j'}-\gamma\|_{C^m(\bar{I};\R^n)}=0 \quad \text{for all $m\geq0$.}
    \]
\end{theorem}

The proof has a slightly different flavor than the fixed-length case.
Here the multipliers are a priori controlled, but instead the length is not.
In particular, the assumption of Theorem \ref{thm:stability_penalized} rules out counterexamples as in Figure \ref{fig:discontinuous_circle} (see also Remark \ref{rem:discontinuous_length_penalized}).

\begin{figure}[htbp]
 \begin{center}
     \includegraphics[scale=0.3]{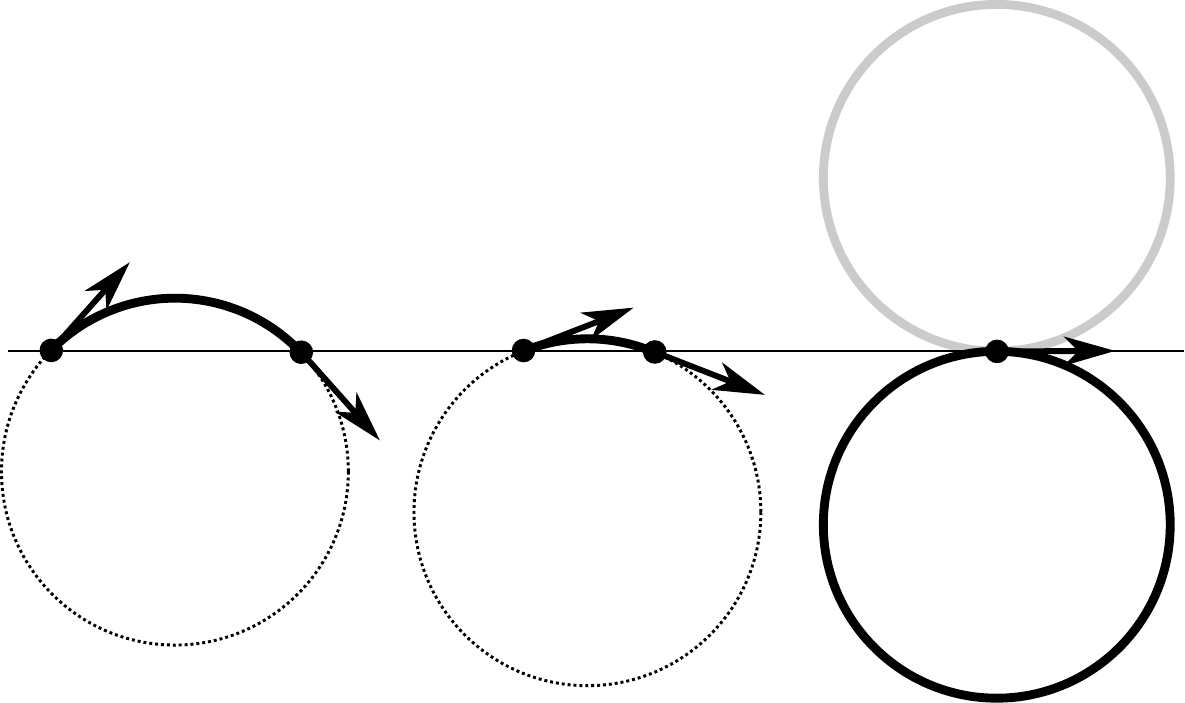}
     \caption{Discontinuous transition of minimizers of $\E_\lambda$.}
     \label{fig:discontinuous_circle}
 \end{center}
\end{figure}

This paper is organized as follows.
In Section \ref{sec:compactness} we prove Theorem \ref{thm:main_compactness} and discuss counterexamples.
In Section \ref{sec:stability} we prove Theorems \ref{thm:stability_fixed} and \ref{thm:stability_penalized}.

\begin{acknowledgements}
    This work grew out of discussions with Georg Dolzmann and Glen Wheeler, whom the author would like to thank.
    This work is supported by JSPS KAKENHI Grant Numbers JP21H00990, JP23H00085, and JP24K00532.
\end{acknowledgements}


\section{Compactness}\label{sec:compactness}

We begin with a standard weak compactness for general bounded sequences.

\begin{lemma}\label{lem:weak_compactness}
    Let $\{\gamma_j\}_{j=1}^\infty\subset W^{2,2}(I;\R^n)$ be a sequence of constant-speed curves satisfying \eqref{eq:main_assumption} and $\sup_{j}|\gamma_j(0)|<\infty$.
    Then $\{\bar{\gamma}_j\}$ contains a subsequence converging to some $\gamma_\infty\in W^{2,2}(I;\R^n)$ in the $W^{2,2}$-weak topology and the $C^1$ topology.
    In particular, $\gamma_\infty$ is a constant-speed curve of positive length.
\end{lemma}

\begin{proof}
    By the boundedness and constant-speed assumptions, we deduce
    \[
    \sup_{j}\|\partial_x\gamma_j\|_{L^2} = \sup_j\L[\gamma_j]<\infty
    \]
    and also
    \[
    \sup_{j}\|\gamma_j\|_{L^2} \leq \sup_{j}\|\gamma_j\|_{L^\infty} \leq \sup_{j}(|\gamma_j(0)|+\L[\gamma_j] ) <\infty.
    \]
    In addition, since
    \[
    \B[\gamma_j]=\int_I|\partial_s^2\gamma_j|^2ds=\frac{1}{\L[\gamma_j]^3}\int_I|\partial_x^2\gamma_j|^2dx,
    \]
    we also have
    \[
    \sup_{j}\|\partial_x^2\gamma_j\|_{L^2}^2 \leq \sup_{j}\L[\gamma_j]^3\B[\gamma_j] <\infty.
    \]
    These imply that $\{\gamma_j\}$ is bounded in $W^{2,2}(I;\R^n)$, thus having a subsequence that $W^{2,2}$-weakly converges to some $\gamma_\infty\in W^{2,2}(I;\R^n)$.
    Compact Sobolev embedding $W^{2,2}(I;\R^n)\subset\subset C^1(\bar{I};\R^n)$ also implies $C^1$-convergence (passing to a subsequence).
    Hence $\gamma_\infty$ must have constant speed $|\partial_x\gamma_\infty|\equiv\mathcal{L}[\gamma_\infty]\geq\inf_{j}\L[\gamma_j]>0$.
\end{proof}

\subsection{Smooth convergence}

Now we discuss when the above weak convergence becomes smooth, assuming that the sequence consists of elasticae.
As explained in the introduction, we look at the behavior of multipliers.
If an elastica solves equation \eqref{eq:EL_elastica} with multiplier $\lambda\in\R$, then we also call it \emph{$\lambda$-elastica} to specify the value of $\lambda$.
Note that $\lambda$ is uniquely determined by a given elastica unless it is a straight segment (for which any $\lambda$ is allowable).

We first prove that if the multipliers are uniformly bounded, then smooth convergence follows by Gagliardo--Nirenberg type interpolation estimates.
This technique is by now standard in the study of elastic flows since Dziuk--Kuwert--Sch\"atzle's pioneering work \cite{DziukKuwertSchatzle2002}.
Here we use results for open elastic flows in \cite{DallAcquaPozzi2014}.
The main advantages are direct applicability to our non-closed elasticae (as stationary solutions) and also their explicit dependence on $\lambda$.

\begin{proposition}\label{prop:bounded_lambda}
    Let $\{\gamma_j\}_{j=1}^\infty\subset C^\omega(\bar{I};\R^n)$ be a sequence of constant-speed $\lambda_j$-elasticae satisfying \eqref{eq:main_assumption} and converging to some curve $\gamma_\infty$ in the $C^1$ topology.
    Suppose that $\sup_j|\lambda_j|<\infty$.
    Then $\{\gamma_j\}$ smoothly converges to $\gamma_\infty$, and $\gamma_\infty$ is an elastica.
\end{proposition}

\begin{proof}
    Since each $\gamma_j$ satisfies the assumption of \cite{DallAcquaPozzi2014}*{Lemma 3.1} (as a stationary solution), the proof of \cite{DallAcquaPozzi2014}*{Lemma 4.4} together with the boundedness assumptions
    \[B:=\sup_{j}\B[\gamma_j]<\infty, \quad L_*:=\inf_j\L[\gamma_j]>0, \quad \Lambda:=\sup_j|\lambda_j|<\infty\]
    implies that for every $m\geq0$ there is $C=C(n,m,B,L_*,\Lambda)>0$ such that 
    \[\sup_j\|\nabla_s^m\kappa_j\|_{L^2}\leq C.\]
    Hence by \cite{DallAcquaPozzi2014}*{Lemma 4.6} we also have (up to redefining $C$)
    \[\sup_j\|\partial_s^m\kappa_j\|_{L^2}\leq C.\]
    Since $\partial_s^m\kappa_j=\partial_s^{m+2}\gamma_j=\L[\gamma_j]^{-(m+2)}\partial_x^{m+2}\gamma_j$, for every $m\geq0$ we obtain
    \[
    \sup_j\|\partial_x^{m+2}\gamma_j\|_{L^2} \leq (L^*)^{m+2}C < \infty,
    \]
    where $L^*:=\sup_j\L[\gamma_j]<\infty$.
    Combining this boundedness with $C^1$-convergence implies smooth convergence (via standard compact Sobolev embeddings).
    In addition, passing to a subsequence, the bounded sequence $\{\lambda_j\}_j$ converges, so by taking $j\to\infty$ in equation \eqref{eq:EL_elastica} for $\gamma_j$ we find that $\gamma_\infty$ is also an elastica.
\end{proof}

The remaining case $\sup_j|\lambda_j|=\infty$ requires a more delicate and ad-hoc understanding of elasticae.
Our proof will be based on explicit formulae for elasticae due to Langer--Singer \cite{LangerSinger1984JLMS} (see also \cite{Singer2008,Miura2024survey}) and their integrability arguments.
We will thus frequently use Jacobi elliptic integrals and functions.
Here we adopt the notation in \cite{Miura2024survey}*{Appendix A}, where one can find all the necessary definitions used in this paper.
In particular, we choose the elliptic parameter $m=p^2\in[0,1]$ instead of the elliptic modulus $p$.

\begin{proposition}\label{prop:unbounded_lambda}
    Let $\{\gamma_j\}_{j=1}^\infty\subset C^\omega(\bar{I};\R^n)$ be a sequence of constant-speed $\lambda_j$-elasticae satisfying \eqref{eq:main_assumption} and converging to some curve $\gamma_\infty$ in the $C^1$ topology.
    Suppose that $\sup_j|\lambda_j|=\infty$.
    Then $\gamma_\infty$ is a straight segment.
\end{proposition}

\begin{proof}
    As $C^1$-convergence is already assumed, it is sufficient to prove that up to taking isometric transformations and subsequences (we will not relabel) the sequence converges to an axis in a certain sense; we will prove that $\gamma_j-\langle\gamma_j,e_n \rangle e_n \to 0$ in $L^2$, where $\{e_i\}_{i=1}^n$ denotes the canonical basis.
    We may assume that (for a subsequence) each $\gamma_j$ is not a straight segment, since otherwise the assertion is trivial.
    
    It is known (see e.g.\ \cite{Miura2024survey}*{Corollary 2.5}) that any elastica is contained in a three-dimensional affine subspace, as the equation is of fourth order.
    Hence, without loss of generality we may assume $n=3$.

    By the known explicit formula for elasticae in $\R^3$ (see e.g.\ \cite{Miura2024survey}*{Theorem 2.12}),
    for each non-straight $\lambda_j$-elastica $\gamma_j:\bar{I}\to\R^3$ there are (unique) parameters $m_j,w_j,A_j,c_j\in\R$ such that
    \begin{equation}\label{eq:parameter_region}
        0\leq m_j\leq w_j\leq 1, \quad w_j>0, \quad A_j> 0,
    \end{equation}
    with the following properties: 
    Let $k_j$ ($=|\partial_s^2\gamma_j|$) be the scalar curvature.
    Let $t_j$ be the torsion of $\gamma_j$; if $\gamma_j$ is planar, then we interpret $t_j\equiv0$.
    Then for some $\beta_j\in\mathbf{R}$,
    \begin{equation}\label{eq:curvature_formula}
      k_j(s)^2=A_j^2\left( 1-\frac{m_j}{w_j}\sn^2\left(\frac{A_j}{2\sqrt{w_j}}s+\beta_j , m_j \right) \right),
    \end{equation}
    where $s\in[0,\L[\gamma_j]]$ denotes the arclength parameter,
    and also
    \begin{equation}\label{eq:det_formula}
        k_j(s)^2t_j(s) = c_j,
    \end{equation}
    \begin{equation}\label{eq:lambda_formula}
      \lambda_j = \frac{A_j^2}{2w_j}(3w_j-m_j-1),
    \end{equation}
    and
    \begin{equation}\label{eq:c_formula}
      4c_j^2 = \frac{A_j^6}{w_j^2}(1-w_j)(w_j-m_j).
    \end{equation}

    In addition, we recall Langer--Singer's integrability results via Killing fields.
    It is known \cite{Singer2008}*{Equation (4)} (and easily verified by directly differentiating) that any elastica $\gamma$ has an associated constant vector field
    \[
    J:=(k^2-\lambda)T + 2\partial_skN + 2ktB \equiv\text{const.},
    \]
    where $T,N,B$ denote the unit tangent, normal, and binormal, respectively (we interpret $2ktB=0$ when $\gamma$ is planar).
    Clearly, $|J|\neq0$ if $k$ is nonconstant (i.e., $\partial_sk\not\equiv0$).
    Also, if $k$ is a nonzero constant, then in terms of the above parameters we have $m=0$, $k^2=A^2$, $\lambda=\frac{3w-1}{2w}A^2$, so $|J|\neq0$ unless $w=1$, i.e., unless $\gamma$ is a circular arc such that $k^2\equiv\lambda$.
    Hence, for our sequence $\{\gamma_j\}$, passing to a subsequence, we have either
    \begin{equation}\label{eq:a_identity}
        a_j^2 := |J_j|^2\equiv (k_j(s)^2-\lambda_j)^2 + 4 \partial_sk_j(s)^2 + 4k_j(s)^2t_j(s)^2 >0,
    \end{equation}
    or each $\gamma_j$ is a circular arc of radius $|\lambda_j|^{-\frac{1}{2}}$.
    The latter case is ruled out by \eqref{eq:main_assumption} and $\sup_j|\lambda_j|=\infty$.
    Therefore, strict positivity \eqref{eq:a_identity} must hold.
    Then, the argument in \cite{Singer2008}*{p.7} implies that (up to isometries) each $\gamma_j$ can be explicitly parametrized in the standard cylindrical coordinate $(r,\theta,z)$.
    In particular, the radius function $r_j=|\gamma_j-\langle \gamma_j,e_3 \rangle e_3|$ is given by 
    \begin{equation}\label{eq:radius}
        r_j(s) = \frac{2}{a_j^2}\sqrt{a_j^2k_j(s)^2-4c_j^2}.
    \end{equation}
    
    Now, in order to prove the assertion, it is sufficient to show that, passing to a subsequence, $r_j\to0$ in a certain (for example $L^2$) sense.
    Since
    \[
    \int_I r_j^2dx = \frac{1}{\L[\gamma_j]} \int_I r_j^2 ds \leq \frac{1}{\L[\gamma_j]}\frac{4}{a_j^2}\int_Ik_j^2ds \leq \frac{4}{a_j^2}\frac{\sup_{j}\B[\gamma_j]}{\inf_j\L[\gamma_j]},
    \]
    it is sufficient to show that
    \[
    \sup_j|a_j|=\infty.
    \]
    Since identity \eqref{eq:a_identity} remains true if we extend the functions $k_j$ and $t_j$ to $\R$ via formulae \eqref{eq:curvature_formula} and \eqref{eq:det_formula}, we can particularly pick a point $s_*\in\R$ at which the right-hand side of \eqref{eq:curvature_formula} takes the maximum, i.e., $k_j(s_*)^2=A_j^2$.
    Inserting this into \eqref{eq:a_identity} implies $a_j^2 \geq (A_j^2-\lambda_j)^2$.
    Hence it is sufficient to show that 
    \begin{equation}\label{eq:A^2-lambda}
        \sup_j |A_j^2-\lambda_j|=\infty.
    \end{equation}
    
    By the assumption $\sup_j|\lambda_j|=\infty$, passing to a subsequence, we have either $\lambda_j\to\infty$ or $\lambda_j\to-\infty$.
    If $\lambda_j\to-\infty$, then clearly $|A_j^2-\lambda_j|\to\infty$.
    In what follows we suppose
    \[
    0\leq\lambda_j\to\infty.
    \]
    
    By \eqref{eq:lambda_formula} we have $\lambda_j \leq \frac{3}{2}A_j^2$ and hence
    \begin{equation}\label{eq:Atoinfty}
        A_j\to\infty.
    \end{equation}
    Also, using positivity $\lambda_j\geq0$ in \eqref{eq:lambda_formula} we deduce $3w_j-m_j-1\geq0$.
    In particular,
    \begin{equation}\label{eq:w_bounded}
        \frac{1}{3}\leq w_j\leq 1.
    \end{equation}
    Then, using $\sn^2+\cn^2=1$, we estimate
    \begin{align*}
        \B[\gamma_j] &=\int_0^{\L[\gamma_j]}A_j^2\left( 1-\frac{m_j}{w_j}\sn^2\left(\frac{A_j}{2\sqrt{w_j}}s+\beta_j , m_j \right) \right) ds\\
        &\geq A_j^2 \int_0^{\L[\gamma_j]}\cn^2\left(\frac{A_j}{2\sqrt{w_j}}s+\beta_j , m_j \right)  ds\\
        &= 2A_j\sqrt{w_j}\int_{\beta_j}^{A_j\L[\gamma_j]/(2\sqrt{w_j})+\beta_j}\cn^2(u, m_j)  du\\
        &\geq \frac{2A_j}{\sqrt{3}}\int_{\beta_j}^{\frac{1}{2}A_j\L[\gamma_j]+\beta_j}\cn^2(u, m_j)  du,
    \end{align*}
    where in the last inequality we used \eqref{eq:w_bounded}.
    By \eqref{eq:main_assumption} and \eqref{eq:Atoinfty}, we need to have
    \begin{equation}\label{eq:integral_limit}
        c_j:=\int_{\beta_j}^{\frac{1}{2}A_j\L[\gamma_j]+\beta_j}\cn^2(u, m_j)  du \to0.
    \end{equation}
    
    Now we prove by contradiction that 
    \begin{equation}\label{eq:m_to_1}
        m_j\to1.
    \end{equation}
    If otherwise, there is a subsequence such that $m_j\to m_\infty\in[0,1)$.
    Then the $2K(m_j)$-periodic integrand $\cn^2(u, m_j)$ locally uniformly converges to $\cn^2(u, m_\infty)$ with finite period $2K(m_\infty)\in[\frac{\pi}{2},\infty)$.
    By periodicity we may assume that $\beta_j\in[-2K(m_j),0]$ and in particular $\{\beta_j\}$ is bounded.
    By \eqref{eq:main_assumption} and \eqref{eq:Atoinfty} we have $A_j\L[\gamma_j]\to\infty$.
    These properties imply that $[0,2K(m_\infty)]\subset[\beta_j,\frac{1}{2}A_j\L[\gamma_j]+\beta_j]$ for any large $j$, so
    \[
    0\leq \int_0^{2K(m_\infty)}\cn^2(u, m_\infty)  du = \lim_{j\to\infty}\int_0^{2K(m_\infty)}\cn^2(u, m_j)  du \leq \lim_{j\to\infty} c_j \overset{\eqref{eq:integral_limit}}{=} 0
    \]
    but this is a contradiction.
    Hence \eqref{eq:m_to_1} follows.

    By \eqref{eq:m_to_1} and $m_j\leq w_j \leq 1$ we also have 
    \begin{equation}\notag
        w_j\to 1.
    \end{equation}
    From this with \eqref{eq:m_to_1} and \eqref{eq:lambda_formula} we deduce that $\frac{\lambda_j}{A_j^2} \to \frac{1}{2}$ and hence, noting \eqref{eq:Atoinfty},
    \[
    |A_j^2-\lambda_j|=A_j^2 \left| 1-\frac{\lambda_j}{A_j^2} \right|\to\infty,
    \]
    which implies \eqref{eq:A^2-lambda}.
    The proof is now complete.
\end{proof}

\begin{proof}[Proof of Theorem \ref{thm:main_compactness}]
    Letting $b_j:=-\gamma_j(0)$, the sequence $\{\bar{\gamma}_j+b_j\}$ satisfies the assumption of Lemma \ref{lem:weak_compactness}, and hence a subsequence $\{\bar{\gamma}_{j'}+b_{j'}\}$ converges to some curve $\bar{\gamma}_\infty$ in the $C^1$-topology.
    
    Now, suppose that $\bar{\gamma}_\infty$ is not a straight segment.
    Then by the contrapositive of Proposition \ref{prop:unbounded_lambda}, we have $\sup_{j'}|\lambda_{j'}|<\infty$, where $\lambda_{j'}$ denotes a multiplier of $\bar{\gamma}_{j'}+b_{j'}$.
    Hence by Proposition \ref{prop:bounded_lambda}, the convergence is smooth and $\bar{\gamma}_\infty$ is an elastica.

    On the other hand, if $\bar{\gamma}_\infty$ is a straight segment, then (convergence may not be smooth but) $\bar{\gamma}_\infty$ is a trivial elastica, so in any case the limit curve $\bar{\gamma}_\infty$ must be an elastica.
    The proof is complete.
\end{proof}

\subsection{Counterexamples}

Finally, we provide counterexamples to smooth convergence in Theorem \ref{thm:main_compactness}.
Of course, the limit curve must be a straight segment.
To this end it is sufficient to construct sequences of constant-speed elasticae $\{\gamma_j\}$ satisfying \eqref{eq:main_assumption} and also
\[
\inf_j\B[\gamma_j]>0.
\]
Indeed, this is equivalent to $\inf_j\|\partial_x^2\gamma_j\|_{L^2}>0$ under \eqref{eq:main_assumption}, so any subsequence cannot converge to the segment smoothly, nor even in the $W^{2,2}$-strong topology.

In fact, we can construct such examples for both $\lambda_j\to-\infty$ and $\lambda_j\to\infty$.
These cases correspond to different phenomena, namely of curvature oscillation and concentration, respectively.


\begin{example}[Curvature oscillation]\label{ex:oscillation}
    We construct an example as in Figure \ref{fig:curvature_oscillation}.
    Consider a sequence $\{\gamma_j\}\subset C^\omega(\bar{I};\R^n)$ of arclength parametrized elasticae such that for all $j\geq1$,
    \[
    \L[\gamma_j]=1.
    \]
    We will choose appropriate parameters in \eqref{eq:curvature_formula}--\eqref{eq:c_formula}.
    Setting $c_j=0$, $w_j=m_j\in(0,1)$, and $\beta_j=0$, we obtain a planar (wavelike) elastica $\gamma_j$ such that
    \[
    k_j^2(s) = A_j^2\cn^2\left(\frac{A_j}{2\sqrt{m_j}}s,m_j\right), \quad \lambda_j=\frac{A_j^2}{2m_j}(2m_j-1).
    \]
    We then compute
    \begin{align}\label{eq:example_bending}
        \B[\gamma_j] &= \int_0^1k_j^2(s)ds = A_j\sqrt{m_j} \int_0^{A_j/\sqrt{m_j}}\cn^2(u,m_j)^2 du.
    \end{align}
    Now, set $A_j=2K(m_j)$ and $m_j=1/j^2$.
    Since $\cn^2(\cdot,m)$ is $2K(m)$-periodic, we compute
    \[
    \B[\gamma_j]= \frac{2K(m_j)}{j}\int_0^{2K(m_j)j} \cn^2(u,m_j)du = 2K(m_j)\int_0^{2K(m_j)} \cn^2(u,m_j)du >0.
    \]
    Since $\lim_{m\to0}\cn(u,m)=\cos{u}$ and $\lim_{m\to0}2K(m)=\pi$ (\cite{Miura2024survey}*{Appendix A}), the bounded convergence theorem implies
    \[
    \lim_{j\to\infty}\B[\gamma_j]= \pi\int_0^{\pi} \cos^2{u}du = \frac{\pi^2}{2}>0.
    \]
    This with $\L[\gamma_j]=1$ implies \eqref{eq:main_assumption} and $\inf_j\B[\gamma_j]>0$.
    Thus we obtain a counterexample.
    For this example we have
    \[
    \lambda_j\to-\infty.
    \]
    Also, since each $\gamma_j$ is planar, the signed curvature is well-defined and given by
    \[
    k_j(s)=2K(m_j)\cn(jK(m_j)s,m_j).
    \] 
    This converges $L^2$-weakly to $0$ in an oscillatory way.
\end{example}

\begin{example}[Curvature concentration]\label{ex:concentration}
    We construct an example as in Figure \ref{fig:curvature_concentration}.
    Again consider arclength parametrized elasticae $\{\gamma_j\}$ such that $\L[\gamma_j]=1$.
    Set $c_j=0$ and $w_j=m_j=1$.
    Then each $\gamma_j$ is a planar (borderline) elastica such that
    \[
    k_j^2(s)=A_j^2\sech^2\left(\frac{A_j}{2}s+\beta_j\right), \quad \lambda_j=\frac{A_j^2}{2}.
    \]
    Now we choose $A_j=2j$, and compute
    \[
    \B[\gamma_j] = 4j\int_{0}^j\sech^2(u+\beta_j)du.
    \]
    Define the continuous map $f_j:\R\to(0,\infty)$ by $f(r):= \int_{0}^j\sech^2(u+r)du$.
    For each $j$ we have $f_j(0)\geq f_1(0)=: c>0$ while $\lim_{r\to\infty}f_j(r) =0$, so there is $r_j\geq 0$ such that $f_j(r_j)=c/j$.
    Setting $\beta_j=r_j$ yields for all $j\geq1$,
    \[
    \B[\gamma_j] = 4c>0.
    \]
    This also implies \eqref{eq:main_assumption} and $\inf_j\B[\gamma_j]>0$, giving a counterexample.
    In this case
    \[
    \lambda_j\to\infty,
    \]
    and the signed curvature is given by
    \[
    k_j(s)=2j\sech(js+r_j).
    \]
    Since $0\leq k_j(s) \leq Cje^{-js}$, the curvature uniformly converges to $0$ on any interval $[\varepsilon,1]$ with small $\varepsilon>0$.
    This example is thus of concentration type.
\end{example}

\section{Stability}\label{sec:stability}

\subsection{Fixed-length case}

Here we prove Theorem \ref{thm:stability_fixed}.
This case directly benefits from the multiplier-free statement of Theorem \ref{thm:main_compactness}.

We first check that smooth compactness follows from Theorem \ref{thm:main_compactness}.

\begin{proposition}\label{prop:clamped_compactness}
    Let $(\Gamma,L)\in \hat{\bA}'$ and $\{(\Gamma_j,L_j)\}_{j=1}^\infty\subset \hat{\bA}$ be a sequence converging to $(\Gamma,L)$.
    Let $\{\gamma_j\}\subset C^\omega(\bar{I};\R^n)$ be a sequence of elasticae such that $\gamma_j\in\A_{\Gamma_j,L_j}$ for $j\geq1$ and $\sup_j\B[\gamma_j]<\infty$.
    Then there is a subsequence of $\{\gamma_j\}$ smoothly converging to an elastica in $\A_{\Gamma,L}$.
\end{proposition}

\begin{proof}
    Since $\sup_j\B[\gamma_j]<\infty$ and $L_j\to L$, the sequence $\{\gamma_j\}$ satisfies \eqref{eq:main_assumption}.
    Then by Theorem \ref{thm:main_compactness} there is a subsequence of $\{\gamma_j\}$ converging to an elastica $\gamma_\infty$ at least in the $C^1$ topology.
    So in particular $\gamma_\infty\in\A_{\Gamma,L}$, and hence $\gamma_\infty$ cannot be a segment.
    By Theorem \ref{thm:main_compactness} the convergence is smooth.
\end{proof}

To improve this convergence to the full limit, an important tool is continuity of the minimal energy.
In fact, we only need upper semicontinuity for later use, but here we address continuity because of its fundamentality.

\begin{lemma}\label{lem:minimal_energy}
    The minimal energy function $m:\hat{\bA}\to(0,\infty)$ defined by
    \[
    m(\Gamma,L):=\min_{\gamma\in\mathcal{A}_{\Gamma,L}} \B[\gamma]
    \]
    is continuous at any $(\Gamma,L)\in\hat{\bA}'$.
\end{lemma}

\begin{proof}
    Fix any $(\Gamma,L)\in \hat{\bA}'$ and any sequence $\{(\Gamma_j,L_j)\}_{j=1}^\infty\subset \hat{\bA}$ converging to $(\Gamma,L)$.
    Fix any minimizer $\gamma\in\mathcal{A}_{\Gamma,L}$ of $\B$.
    Then $m(\Gamma,L)=\B[\gamma]$.

    We first prove upper semicontinuity.
    Since $\gamma$ has nonzero curvature by $(\Gamma,L)\not\in\hat{\bA}_s$, and since $\frac{d}{d\varepsilon}\L[\gamma_\varepsilon] = -\int_I \langle \kappa,\eta \rangle ds $ for $\gamma_\varepsilon:=\gamma+\varepsilon\eta$ with $\eta\in C^\infty_c(I;\R^n)$, we can find a suitable $\eta$ such that $\varepsilon\mapsto\L[\gamma_\varepsilon]$ is strictly increasing on $[-\varepsilon_\eta,\varepsilon_\eta]$ for some small $\varepsilon_\eta>0$.
    Since $\eta$ is supported on a compact set $K\subset I$, by suitably deforming $\gamma_\varepsilon$ around the endpoints, namely on $\bar{I}\setminus K$, we can find $\gamma_{\varepsilon,j}\in\mathcal{A}_{\Gamma_j,L_j}$ such that for any $\varepsilon\in[-\varepsilon_\eta,\varepsilon_\eta]$,
    \begin{equation}\label{eq:length_difference_minconti}
        \L[\gamma_{\varepsilon,j}] - \L[\gamma_\varepsilon] = \L[\gamma_{0,j}] - L,
    \end{equation}
    and since $\Gamma_j\to\Gamma$, we can also have $\gamma_{\varepsilon,j}\to\gamma_\varepsilon$ in $C^2$ so that
    \begin{align}
        \lim_{j\to\infty}\L[\gamma_{\varepsilon,j}] &=\L[\gamma_\varepsilon], \label{eq:convegence_minconti}\\ 
        \lim_{j\to\infty}\B[\gamma_{\varepsilon,j}] &=\B[\gamma_\varepsilon]. \label{eq:convegence2_minconti}
    \end{align}
    (See Remark \ref{rem:deformation} for more details.)
    Then, by $L_j\to L$ and \eqref{eq:convegence_minconti} (for $\varepsilon=0$), there is a large $j_\eta$ such that for any $j\geq j_\eta$,
    \[
    a_j:=\max\{|L_j-L|,|\L[\gamma_{0,j}] - L|\} \leq \frac{1}{2}\min\{ \L[\gamma_{\varepsilon_\eta}] - \L[\gamma], \L[\gamma]-\L[\gamma_{-\varepsilon_\eta}] \}.
    \]
    This with \eqref{eq:length_difference_minconti} and the strict monotonicity of $\L[\gamma_\varepsilon]$ implies that for any $j\geq j_\eta$ there exists $\varepsilon_j\in[-\varepsilon_\eta,\varepsilon_\eta]$ such that $\L[\gamma_{\varepsilon_j,j}] = L_j$.
    Hence $\gamma_{\varepsilon_j,j} \in \mathcal{A}_{\Gamma_j,L_j}$.
    Note also that $\varepsilon_j\to0$ since $a_j\to0$.
    Therefore, by \eqref{eq:convegence2_minconti} and $\lim_{\varepsilon\to0}\B[\gamma_\varepsilon]=\B[\gamma]$,
    \[
    \limsup_{j\to\infty} m(\Gamma_j,L_j) \leq \lim_{j\to\infty}\B[\gamma_{\varepsilon_j,j}] = \B[\gamma] = m(\Gamma,L).
    \]
    This is the desired upper semicontinuity.

    Next we prove lower semicontinuity.
    Let $\{\gamma_j\}$ be a sequence of elasticae (minimizers) such that $\B[\gamma_j]=m(\Gamma_j,L_j)$.
    The upper semicontinuity of $m$ particularly implies $\sup_{j}\B[\gamma_j]<\infty$.
    Hence by Proposition \ref{prop:clamped_compactness} any subsequence has a subsequence $\{\gamma_{j'}\}$ smoothly converging to an elastica $\gamma_\infty\in\mathcal{A}_{\Gamma,L}$.
    Therefore,
    \[
    \lim_{j'\to\infty}\B[\gamma_{j'}] = \B[\gamma_\infty] \geq m(\Gamma,L),
    \]
    and hence even for the full sequence,
    \[
    \liminf_{j\to\infty}m(\Gamma_j,L_j) = \liminf_{j\to\infty}\B[\gamma_j]\geq m(\Gamma,L).
    \]
    This combined with upper semicontinuity completes the proof.
\end{proof}

\begin{remark}
    The minimal energy function is not upper semicontinuous on $\hat{\bA}$.
    In fact, for any $(\Gamma,L)\in \hat{\bA}_s=\hat{\bA}\setminus\hat{\bA}'$ we have $m(\Gamma,L)=0$, but there is a sequence $\{(\Gamma_j,L_j)\}\subset \hat{\bA}'$ converging to $(\Gamma,L)$ such that $m(\Gamma_j,L_j)\to\infty$.
    For example, under \eqref{eq:generic_angle} the minimal energy diverges in the limit $L\to|P_1-P_0|$ \cite[Section 6]{Miura20}.
\end{remark}

\begin{remark}\label{rem:deformation}
    Here we briefly give a more precise argument for the deformation of a curve near the endpoints.
    Let $\Gamma_j:=(P_0^j,P_1^j,V_0^j,V_1^j)$.
    Define $\gamma_{\varepsilon,j}$ (in the proof of Lemma \ref{lem:minimal_energy}) to be the constant-speed reparametrization of
    \begin{align*}
        \hat{\gamma}_{\varepsilon,j} &:= (1-\zeta(\tfrac{x}{\delta_j}))\gamma_\varepsilon(x) + \zeta(\tfrac{x}{\delta_j})(P_0^j+xV_0^j+\tfrac{1}{2}x^2\gamma_\varepsilon''(0))\\
        &\qquad + (1-\zeta(\tfrac{1-x}{\delta_j}))\gamma_\varepsilon(x) + \zeta(\tfrac{1-x}{\delta_j})(P_1^j+(x-1)V_1^j+\tfrac{1}{2}(x-1)^2\gamma_\varepsilon''(1))
    \end{align*}
    with a smooth cut-off function $\zeta:\R\to\R$ such that $\zeta|_{(-\infty,0]}\equiv1$ and $\zeta|_{[1,\infty)}\equiv0$ and with a slowly decaying sequence $\delta_j\to+0$ (depending on the rate of $\Gamma_j\to\Gamma$).
    Since $\hat{\gamma}_{\varepsilon,j}=\gamma_\varepsilon$ on $[\delta_j,1-\delta_j]$, the deformation is done outside the fixed compact set $K\subset I$ for large $j$.
    It is sufficient to show that $\hat{\gamma}_{\varepsilon,j}$ converges in $C^2$ to the constant-speed curve $\gamma_\varepsilon$.
    For simplicity we focus on the convergence near the origin, or just assume $P_1^j=P_1$ and $V_1^j=V_1$.
    Also we only compute the most delicate second-order derivatives:
    \begin{align*}
        |\hat{\gamma}_{\varepsilon,j}''(x)-\gamma_\varepsilon''(x)|
        &\leq \tfrac{1}{\delta_j^2}|\zeta''(\tfrac{x}{\delta_j})||\gamma_\varepsilon(x)-(P_0^j+xV_0^j+\tfrac{1}{2}x^2\gamma_\varepsilon''(0))|
        \\
        & \qquad + \tfrac{1}{\delta_j}|\zeta'(\tfrac{x}{\delta_j})||\gamma_\varepsilon'(x)-(V_0^j+x\gamma_\varepsilon''(0))| \\
        & \qquad\qquad + |\zeta(\tfrac{x}{\delta_j})||\gamma_\varepsilon''(x)-\gamma_\varepsilon''(0)|\\
        & \leq \tfrac{1}{\delta_j^2}|\zeta''(\tfrac{x}{\delta_j})|(o(x^2)+|P_0^j-P_0|+x|V_0^j-V_0|)
        \\
        & \qquad + \tfrac{1}{\delta_j}|\zeta'(\tfrac{x}{\delta_j})|(o(x)+|V_0^j-V_0|) + |\zeta(\tfrac{x}{\delta_j})|o(1),
    \end{align*}
    where the little-o notation is about the limit $x\to0$.
    Letting $C:=\|\zeta\|_{C^2}$,
    \begin{align*}
        & \sup_{x\in[0,1]}|\hat{\gamma}_{\varepsilon,j}''(x)-\gamma_\varepsilon''(x)|\\
        & \leq C\left(\frac{o(\delta_j^2)+|P_0^j-P_0|+\delta_j|V_0^j-V_0|}{\delta_j^2}+ \frac{o(\delta_j)+|V_0^j-V_0|}{\delta_j} + o(1)   \right),
    \end{align*}
    so the RHS converges to $0$ if $\delta_j\to0$ while $\delta_j^{-2}|P_0^j-P_0|\to0$ and $\delta_j^{-1}|V_0^j-V_0|\to0$.
\end{remark}

We are now in a position to complete the proof of Theorem \ref{thm:stability_fixed}.

\begin{proof}[Proof of Theorem \ref{thm:stability_fixed}]
    Lemma \ref{lem:minimal_energy} implies that 
    \[
    \lim_{j\to\infty}\B[\gamma_j]= m(\Gamma,L).
    \]
    On the other hand, by Proposition \ref{prop:clamped_compactness}, there is a subsequence $\{\gamma_{j'}\}\subset\{\gamma_j\}$ that smoothly converges to some elastica $\gamma\in\mathcal{A}_{\Gamma,L}$.
    Since
    \[
    \lim_{j'\to\infty}\B[\gamma_j]= \B[\gamma],
    \]
    we have $\B[\gamma]=m(\Gamma,L)$ and hence $\gamma$ is a minimizer of $\B$ in $\mathcal{A}_{\Gamma,L}$.
\end{proof}


\subsection{Length-penalized case}

Finally we prove Theorem \ref{thm:stability_penalized}.
In this part the main subtlety lies in the control of length from below.

As in the fixed-length case, we begin with smooth compactness.
Here we need a new assumption in order to avoid length-vanishing sequences.

\begin{proposition}\label{prop:clamped_compactness_2}
    Let $(\Gamma,\lambda)\in\X'\times(0,\infty)$ and $\{(\Gamma_j,\lambda_j)\}_{j=1}^\infty\subset \X\times(0,\infty)$ be a sequence converging to $(\Gamma,\lambda)$.
    Let $\{\gamma_j\}\subset C^\omega(\bar{I};\R^n)$ be a sequence of $\lambda_j$-elasticae such that $\gamma_j\in\A_{\Gamma_j}$ for $j\geq1$ and $\sup_j\E_{\lambda_j}[\gamma_j]<\infty$.
    Then there is a subsequence of $\{\gamma_j\}$ smoothly converging to a $\lambda$-elastica in $\A_{\Gamma}$.
\end{proposition}

\begin{proof}
    By $\sup_{j}\E_{\lambda_j}[\gamma_j]<\infty$ and $\lambda_j\to\lambda$ we deduce that $\sup_j\B[\gamma_j]<\infty$ and $\sup_{j}\L[\gamma_j]<\infty$.
    Now we prove $\inf_{j}\L[\gamma_j]>0$.
    If $|P_1-P_0|>0$, then clearly $\L[\gamma_j]\geq|P_1-P_0|>0$, so we suppose that $|P_1-P_0|=0$.
    Then by $\Gamma\not\in \X_c$ we have $V_0\neq V_1$, and hence by the Cauchy--Schwarz inequality,
    \[
     \sqrt{\L[\gamma_j]\B[\gamma_j]} \geq \int_I|\kappa_j|ds \geq \left|\int_I\kappa_jds\right| = |\partial_s\gamma_j(1)-\partial_s\gamma_j(0)| \to |V_1-V_0|>0.
    \]
    Since $\sup_{j}\B[\gamma_j]<\infty$, in this case we also have $\inf_{j}\L[\gamma_j]>0$.

    Therefore, $\{\gamma_j\}$ satisfies \eqref{eq:main_assumption}.
    By Theorem \ref{thm:main_compactness} there is a subsequence converging to some elastica $\gamma_\infty$ in the $C^1$ topology.
    Then $\gamma_\infty\in\A_\Gamma$, which is not a segment since $\Gamma\not\in\X_s$, so by Theorem \ref{thm:main_compactness} the convergence is smooth.
    Letting $j\to\infty$ in equation \eqref{eq:EL_elastica} for $\gamma_j$ implies that $\gamma_\infty$ has the multiplier $\lambda$.
\end{proof}

\begin{remark}\label{rem:discontinuous_length_penalized}
    The assumption $\Gamma\not\in \X_c$ is also unremovable.
    Let $(\Gamma,\lambda)\in \X_c\times(0,\infty)$.
    If we define $\Gamma_j=(P_0^j,P_1^j,V_0^j,V_1^j):=(P_0,P_0+\frac{1}{j}e,e,e)$ for some $e\in\S^{n-1}$, then $(\Gamma_j,\lambda)\to(\Gamma,\lambda)$, but a unique minimizer of $\E_\lambda$ in $\A_{\Gamma_j}$ is a segment with vanishing length as $j\to\infty$.
    Although now $\Gamma_j\not\in\X'$, we can even perturb $V_0^j,V_1^j$ so that $\Gamma_j\in\X'$ and $\A_{\Gamma_j}$ admits a (small) circular arc of radius $1/\sqrt{\lambda}$, which is also a minimizer with vanishing length, as in Figure \ref{fig:discontinuous_circle}.
\end{remark}



    
    

Before discussing continuity of the minimal energy, we need to ensure existence of minimizers.
This follows by a standard direct method.

\begin{theorem}\label{thm:existence_penalized}
    Let $(\Gamma,\lambda)\in\X\times(0,\infty)$.
    Then there is a minimizer of $\E_\lambda$ in $\A_\Gamma$.
\end{theorem}

\begin{proof}    
    It is clear that $\A_\Gamma\neq\emptyset$ and $\E_\lambda$ is nonnegative.
    Let $\{\gamma_j\}\subset\A_\Gamma$ be a minimizing sequence of $\E_\lambda$, that is, $\E_\lambda[\gamma_j]\to \inf_{\A_\Gamma}\E_\lambda\geq0$.
    In particular, we have $\sup_j\B[\gamma_j]<\infty$ and $\sup_j\L[\gamma_j]<\infty$.

    We also have the non-degeneracy of the length
    \[\inf_j\L[\gamma_j]>0.\]
    Indeed, if $\Gamma\not\in\X_c$, then the non-degeneracy follows by the same reasoning as in the proof of Proposition \ref{prop:clamped_compactness_2}.
    In the case where $\Gamma \in \X_c$, each $\gamma_j \in \A_\Gamma$ can be regarded as a closed $W^{2,2}$-curve. 
    Therefore, by applying the Cauchy--Schwarz inequality and Fenchel's theorem, we obtain
    $
    \L[\gamma_j]\B[\gamma_j] \geq (\int_{\gamma_j}|\kappa|ds )^2 \geq 4\pi^2.
    $
    Since $\sup_j \B[\gamma_j] < \infty$, this also implies the desired non-degeneracy.
    
    Then by Lemma \ref{lem:weak_compactness} we deduce that, after passing to a subsequence with constant-speed reparametrizations (which we do not relabel), the minimizing sequence converges to some $\gamma_\infty\in W^{2,2}(I;\R^n)$ in the $W^{2,2}$-weak and $C^1$ sense.
    This implies that $\gamma_\infty \in \A_\Gamma$ and $\L[\gamma_j]\to\L[\gamma_\infty]$.
    Furthermore, by weak lower semicontinuity, we obtain
    \[
    \liminf_{j\to\infty}\B[\gamma_j] = \liminf_{j\to\infty}\frac{1}{\L[\gamma_j]^3}\|\partial_x^2\gamma_j\|_2^2 \geq \frac{1}{\L[\gamma_\infty]^3}\|\partial_x^2\gamma_\infty\|_2^2  = \B[\gamma_\infty].
    \]
    Thus $\gamma_\infty \in \A_\Gamma$ satisfies $\inf_{\A_\Gamma} \E_\lambda = \lim_{j \to \infty} \E_\lambda[\gamma_j] \geq \E_\lambda[\gamma_\infty]$, indicating that $\gamma_\infty$ is the desired minimizer.
\end{proof}

Now we proceed with continuity of the minimal energy.
Here we again need to do away with $\X_c$ to avoid discontinuity as in Remark \ref{rem:discontinuous_length_penalized} and Figure \ref{fig:discontinuous_circle}.

\begin{lemma}\label{lem:minimal_energy_2}
    The minimal energy function $\tilde{m}:\X\times(0,\infty)\to(0,\infty)$ defined by
    \[
    \tilde{m}(\Gamma,\lambda):=\min_{\gamma\in\mathcal{A}_{\Gamma}} \E_\lambda[\gamma]
    \]
    is upper semicontinuous at any $(\Gamma,\lambda)\in(\X'\cup\X_c)\times(0,\infty)$.
    In addition, $\tilde{m}$ is continuous at any $(\Gamma,\lambda)\in \X'\times(0,\infty)$.
\end{lemma}

\begin{proof}
    The proof is almost parallel to that of Lemma \ref{lem:minimal_energy}.
    Using existence of minimizers (Theorem \ref{thm:existence_penalized}), we can prove upper semicontinuity at each $(\Gamma,\lambda)\in\X'\times(0,\infty)$ by the same deformation argument; it is in fact much easier as we need not prescribe the length.
    Lower semicontinuity also follows by the same compactness argument, for which we use Proposition \ref{prop:clamped_compactness_2} (instead of Proposition \ref{prop:clamped_compactness}) and this is why we need to assume $\Gamma\not\in\X_c$.
\end{proof}

\begin{proof}[Proof of Theorem \ref{thm:stability_penalized}]
    The assertion follows by a completely parallel argument to the proof of Theorem \ref{thm:stability_fixed}, where we replace Proposition \ref{prop:clamped_compactness} with Proposition \ref{prop:clamped_compactness_2} and Lemma \ref{lem:minimal_energy} with Lemma \ref{lem:minimal_energy_2}.
\end{proof}

\bibliography{elastica}

\end{document}